\documentclass[11 pt]{amsart}
\usepackage{amsmath}
\usepackage{amsthm}
\usepackage{amsfonts}
\usepackage{graphicx}

\newtheorem{theorem}{Theorem}[section]

\theoremstyle{definition}

\theoremstyle{remark}

\numberwithin{equation}{section}

\begin{document}

\title[On the area of minimal surfaces in a slab]
{On the area of minimal surfaces in a slab}
\author[J. CHOE and B. Daniel]{JAIGYOUNG CHOE and BENO\^{I}T DANIEL}
\date{}
\thanks{J.C. supported in part by NRF 2011-0030044, SRC-GAIA. B.D. supported in part by the ANR-11-IS01-0002 grant.}
\address{Korea Institute for Advanced Study, Seoul, 130-722, Korea}
\email{choe@kias.re.kr}

\address{Universit\'e de Lorraine\\
Institut \'Elie Cartan de Lorraine\\
UMR 7502\\
CNRS\\
B.P. 70239\\
F-54506 Vand\oe{}uvre-l\`es-Nancy cedex\\
France\\
and Korea Institute for Advanced Study\\
Seoul\\
130-722\\
Korea}
\email{benoit.daniel@univ-lorraine.fr}

\date{}


\begin{abstract}
Consider a non-planar orientable minimal surface $\Sigma$ in a slab which is possibly with genus or with more than two boundary components. We show that there exists a catenoidal waist $\mathcal{W}$ in the slab whose flux has the same vertical component as $\Sigma$ such that ${\rm Area}(\Sigma)\geq{\rm Area}(\mathcal{W})$, provided the intersections of $\Sigma$ with horizontal planes have the same orientation.
\end{abstract}

\maketitle

\section{Introduction}

The catenoid is the first nontrivial minimal surface discovered. It was Euler who found it in 1744 in the process of proving that when the catenary is rotated about an axis it generates a surface of smallest area \cite{3}. In 1860 Bonnet showed that the catenoid is the only nonplanar minimal surface of revolution \cite{2}. More recently Schoen also characterized the catenoid as the unique complete minimal surface with finite total curvature and with two embedded ends \cite{8}. Similarly Lopez and Ros showed that the catenoid is the only complete embedded nonplanar minimal surface of finite total curvature and genus zero \cite{5}. Moreover, they proved that the catenoid is the unique complete embedded minimal surface in $\mathbb R^3$ of Morse index one \cite{6}. Finally, Collin \cite{collin} characterized the catenoid as the unique properly embedded minimal annulus.

More recently, rotationally symmetric compact pieces of the catenoid have been characterized as minimal annuli in a slab in two ways. Pyo showed that the catenoid in a slab is the only minimal annulus meeting the boundary of the slab in a constant angle \cite{7}. Also it was proved by Bernstein and Breiner that all embedded minimal annuli in a slab have area bigger than or equal to the minimum area of the catenoids in the same slab \cite{1}. That minimum is attained by the catenoidal waist along the boundary of which the rays from the center of the slab are tangent to the waist. This waist is said to be maximally stable because its proper subset is stable and any subset of the catenoid properly containing the waist is unstable (see Proposition 1, \cite{4}).

It is tempting to conjecture that Bernstein-Breiner's theorem should also hold for an immersed minimal surface possibly with genus and/or with more than two boundary components in a slab as they conjectured. In this paper we prove this conjecture provided the minimal surface is orientable and the intersections of the minimal surface with horizontal planes have the same orientation. The orientation of a horizontal section is induced by the surface so that its conormal is upward at regular points. If the minimal surface in a slab has more than two boundary components or nonzero genus, it may happen that a horizontal section of the surface consists of two or more closed curves which have both clockwise and counterclockwise orientations. It may also happen that a horizontal section contains a curve of rotation number zero, since the surface is not assumed to be embedded. If none of this happens, then we can prove the conjecture. Figure 1 shows the horizontal sections $\Sigma\cap P$ with the same orientation and with the mixed orientations.
\begin{center}
\includegraphics[width=5in]{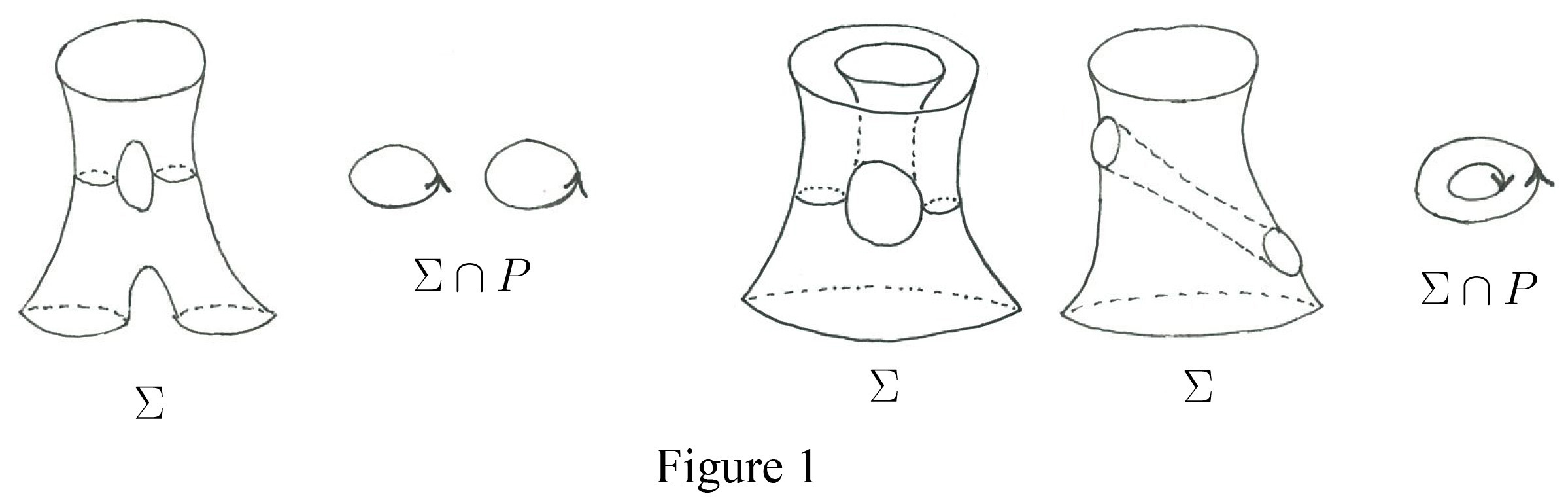}\\
\end{center}

For their proof Bernstein and Breiner used Osserman-Schiffer's theorem \cite{10} that the length $L(u)$ of the curve $\{u={\rm const}\}$ for a harmonic function $u$ on a minimal surface $\Sigma$ satisfies $$L''(u)\geq L(u),$$
where equality holds if and only if $\Sigma$ is the catenoid or an annulus in $\mathbb R^2$. By comparison,
we directly compute the area of $\Sigma$ using the conformal metric $\mathrm{d}s^2=\cosh^2\kappa(u,v)(\mathrm{d}u^2+\mathrm{d}v^2)$ on $\Sigma$, take the average of $\kappa(u,v)$ along $\{u={\rm const}\}$ and use the convexity of the function $\cosh^2$.

At the end of this paper, we propose some open problems in relation to our theorem.

\section{Theorem}

Given the catenoid $\mathcal{C}=\{(x,y,z)\in\mathbb R^3:\cosh z=\sqrt{x^2+y^2}\}$, $\mathcal{C}^b_a:=\mathcal{C}\cap\{a\leq z\leq b\}$ is called a catenoidal waist. If there exists a point $p=(0,0,c)$ such that the rays emanating from $p$ are tangent to $\mathcal{C}$ along the boundary circles $\mathcal{C}\cap\{z=a,b\}$, then $\mathcal{C}^b_a$ is called a {\it maximally stable} waist. This is because the homotheties centered at $p$ give a foliation of a tubular neighborhood of $\mathcal{C}^b_a$, which generates a Jacobi field $J$ on $\mathcal{C}^b_a$ with $|J|>0$ and vanishing only on $\partial\mathcal{C}^b_a$. Let $\beta>0$ be the unique solution to the equation
\begin{equation}\label{beta}
\tanh z=1/z.
\end{equation}
Then it is easy to see that the tangent to the graph of $r=\cosh z$ at $z=\beta$ passes through the origin. It follows that $\mathcal{C}^\beta_{-\beta}$ is maximally stable. The image of a catenoidal waist (respectively, a maximally stable waist) by a translation or a homothety will still be called a catenoidal waist (respectively, a maximally stable waist).

\begin{theorem}\label{theorem}
Given a horizontal slab $\mathrm H_{-a}^a:=\{-a\leq z\leq a\}$ in $\mathbb R^3$, let $\Sigma$ be a non-planar orientable compact immersed minimal surface in $\mathrm H^a_{-a}$ with $\partial\Sigma\subset\partial\mathrm H^a_{-a}$ such that all the components of its intersection with any horizontal plane have the same orientation, that is, their rotation numbers are either all positive or all negative. Then there exists a catenoidal waist $\mathcal{W}\subset\mathrm H^a_{-a}$ whose flux has the same vertical component as $\Sigma$ such that
\begin{equation}\label{area}
{\rm Area}(\Sigma)\geq{\rm Area}(\mathcal{W}).
\end{equation}
Also
\begin{equation}\label{area2}
{\rm Area}(\mathcal{W})\geq{\rm Area}\left(\frac{a}{\beta}\,\mathcal{C}^\beta_{-\beta}\right),
\end{equation}
where $\beta$ satisfies $\tanh\beta=\frac{1}{\beta}$ and $\frac{a}{\beta}\mathcal{C}^\beta_{-\beta}$ is the homothetic expansion of the catenoid $\mathcal{C}^\beta_{-\beta}$ by the factor of $\frac{a}{\beta}$. The boundary circles of $\frac{a}{\beta}\,\mathcal{C}^\beta_{-\beta}$ lie on the boundary of $\mathrm H^a_{-a}$ and $\frac{a}{\beta}\,\mathcal{C}^\beta_{-\beta}$ is maximally stable. Moreover,
$${\rm Area}(\Sigma)={\rm Area}\left(\frac{a}{\beta}\,\mathcal{C}^\beta_{-\beta}\right)$$
if and only if $\Sigma=\frac{a}{\beta}\,\mathcal{C}^\beta_{-\beta}$ up to translation.
\end{theorem}

\begin{proof}
The minimality of $\Sigma$ in $\mathbb R^3$ implies that the Euclidean coordinates $x,y,z$ of $\mathbb R^3$ are harmonic on $\Sigma$. The critical points of $x,y,z$ are isolated on $\Sigma$. Let $u=z|_\Sigma$ and define $v=u^*$, the harmonic conjugate of $u$ on $\Sigma$. Note that $v$ is multivalued on $\Sigma$ but $\mathrm{d}v$ is well defined there. Let $\tilde{\Sigma}$ be the set of regular points of $z$; then $w=u+iv$ is a local complex parameter on $\tilde{\Sigma}$. The Gauss map $g:\Sigma\rightarrow \mathbb C\cup\{\infty\}$ is a meromorphic function which is used to express the metric of $\Sigma$:
$$\mathrm{d}s^2=\cosh^2\kappa|\mathrm{d}w|^2,\,\,\,\kappa=\ln|g|={\rm Re}(\log g),\,\,\,\cosh\kappa=\frac{1}{2}\left(|g|+\frac{1}{|g|}\right).$$
The key idea of the proof is to take the average of the harmonic function $\kappa(u,v)$ along the level curves of $u$, $\gamma_c:\{u={\rm c}\}$. To do so, we need to find the total variation of $v$ along $\gamma_c$:
\begin{equation}\label{f}
\int_{\gamma_c}\mathrm{d}v=\int_{\gamma_c}\mathrm{d}z^*:=-f(c).
\end{equation}
The unit conormal $\nu$ to $\gamma_c$ on $\Sigma$ defines the {\it flux} of $\Sigma$ along $\gamma_c$ as follows:
$$\int_{\gamma_c}\nu\mathrm{d}s=\int_{\gamma_c}\left(\begin{array}{c}\mathrm{d}x(\nu)\\\mathrm{d}y(\nu)\\\mathrm{d}z(\nu)\end{array}\right)\mathrm{d}s
=-\int_{\gamma_c}\left(\begin{array}{c}\mathrm{d}x^*(\tau)\\\mathrm{d}y^*(\tau)\\\mathrm{d}z^*(\tau)\end{array}\right)\mathrm{d}s,$$
where $\tau$ is the unit tangent to $\gamma_c$, i.e., $-90^\circ$-rotation of $\nu$ on $\Sigma$. Hence $f(c)$ equals the vertical component of the flux of $\Sigma$ along $\gamma_c$, which is in fact constant for $-a\leq c\leq a$. From now on, we simply denote this constant by $f$. This constant is positive, since the orientation of $\gamma_c$ is chosen such that the conormal is upward.

The average $h(u)$ of $\kappa(u,v)$ along $\gamma_u$ is defined by
$$h(u)=-\frac{1}{f}\int_{\gamma_u}\kappa(u,v)\mathrm{d}v.$$
It should be remarked that the average is taken over all the components of $\gamma(u)$ and so that the topology of $\Sigma$ is irrelevant to $h(u)$. Since $\kappa(u,v)$ is harmonic away from the critical points of $u$, so is $h(u)$ away from the critical values of $u$:
$$h''=\Delta h=0.$$
Hence $h(u)$ is linear in an open interval where $u$ is regular. Let's compute the slope of $h(u)$. By the Cauchy-Riemann equations,
$$h'(u)=-\frac{1}{f}\int_{\gamma_u}\kappa_u(u,v)\mathrm{d}v=-\frac{1}{f}\int_{\gamma_u}\kappa_v^*(u,v)\mathrm{d}v.$$
Since
$$\log g=\ln|g|+i\arg g=\kappa(u,v)+i\,\kappa^*(u,v),$$ we have $\kappa^*=\arg g$ and so
$\int_{\gamma_u}\kappa_v^*(u,v)\mathrm{d}v$ equals the total variation of $\arg g$ on $\gamma_u$, i.e., $2\pi$ times the total rotation number $r(\gamma_u)$ of the set $\gamma_u$ which is the union of a finite number of closed curves.
Hence
$$h'(u)=-\frac{2\pi \cdot r(\gamma_u)}{f}.$$


By our hypothesis, $\gamma_u$ consists of $n$ closed curves with the same orientation, which we may assume all clockwise. Hence \begin{equation}\label{r>1}r(\gamma_u)\leq-n\leq-1.\end{equation} Moreover this rotation number is constant on any interval without critical value of the height function.

The function $h$ is piecewise linear and $h'$ is a step function. We have $\log g(w)=\pm\infty$ at the points where the tangent plane to $\Sigma$ is horizontal, i.e., $g(w)=0, \infty$. We recall that these points are isolated. So $h'(u)$ is not defined at the height $u$ of the horizontal points. Even so, $h(u)$ is continuous at this height. This can be proved as follows.

We claim that the $1$-form $\kappa(u,v)\mathrm{d} v$ extends continuously at a horizontal point $p$. Let $Z$ be a complex coordinate around $p$ such that $Z(p)=0$, $$\frac{\mathrm{d}w}{\mathrm{d}Z}=\mathrm{O}(Z^m)\quad\textrm{as}\quad Z\to0$$ and $$g(Z)=Z^m\quad\textrm{or}\quad g(Z)=Z^{-m}$$ for some positive integer $m$. Hence
$$\kappa(u,v)\mathrm{d} v=\mathrm{Im}(\log|g(w)|\mathrm{d} w)=\mathrm{O}(Z^m\log|Z|\mathrm{d} Z).$$ This proves the claim. And since $\gamma_u$ depends continuously on $u$, this proves that $h$ is continuous at the height of a horizontal point.


We now compute the area of $\Sigma$:
$${\rm Area}(\Sigma)=\int_{-a}^a\int_{\gamma_u}\cosh^2\kappa(u,v)\mathrm{d}v\,\mathrm{d}u\geq\int_{-a}^af\cosh^2h(u)\mathrm{d}u,$$
where we have the inequality due to the convexity of the function $\cosh^2$. Recall that all the components of $\gamma_u$ are assumed to have the same orientation. Hence \eqref{r>1} implies that $h(u)$ is an increasing function with slope at least $2\pi/f$. If $h(u)$ vanishes at some height $d$, define
$$k(u)=\frac{2\pi}{f}(u-d).$$
(See Figure 2.) If $h(u)$ has no zero and ${\rm min}\,h(u)=h(-a)>0$, define
$$k(u)=\frac{2\pi}{f}(u+a)+h(-a),$$
and if ${\rm max}~h(u)=h(a)<0$, define
$$k(u)=\frac{2\pi}{f}(u-a)+h(a).$$
\begin{center}
\includegraphics[width=2in]{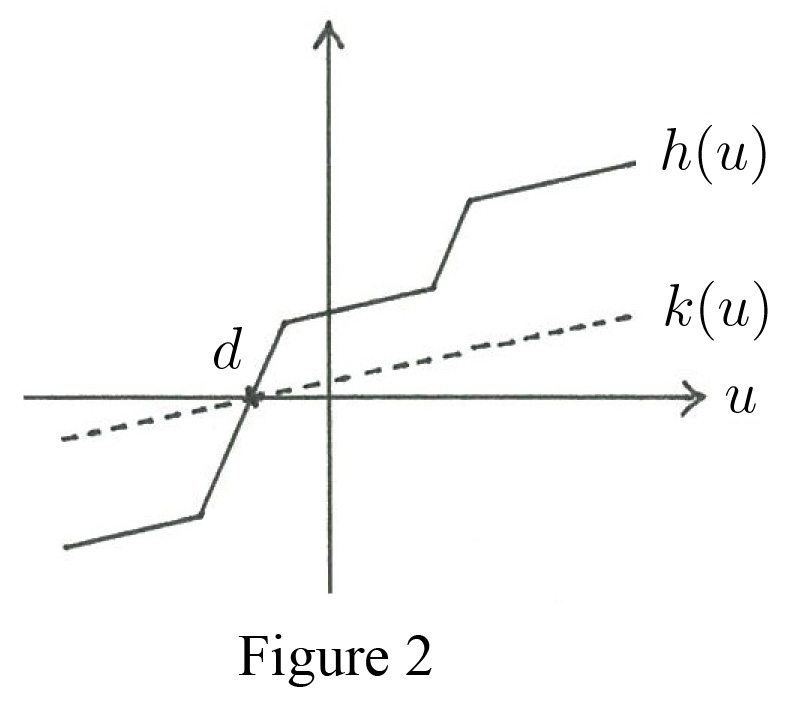}\\
\end{center}
Then we have for every $u$
$$h(u)\leq k(u)\leq0~~~~\,\,\,{\rm or}~~~~\,\,\,0\leq k(u)\leq h(u).$$
It follows that
$$\cosh h(u)\geq\cosh k(u).$$
Therefore
\begin{equation}\label{hk}
{\rm Area}(\Sigma)\geq\int_{-a}^af\cosh^2h(u)\mathrm{d}u\geq\int_{-a}^af\cosh^2k(u)\mathrm{d}u.
\end{equation}

By the way, for some $d_0\in\mathbb R$
$$k(u)=\frac{2\pi}{f}u+d_0={\rm Re}\left(\frac{2\pi}{f}w+d_0\right).$$
Consider the Weierstrass data on $(\mathbb C/if\mathbb Z,w)$:
$$G=\exp\left({\frac{2\pi}{f}w+d_0}\right) \,\,{\rm and}\,\, {\rm the}\,\, 1{\rm -form}\,\,\frac{1}{G(w)}\mathrm{d}w.$$
These data give rise to the catenoid which we denote as $\mathcal{C}(2\pi/f;d_0)$.
So
\begin{eqnarray*}{\rm Area}(\Sigma)&\geq&\int_{-a}^af\cosh^2 k(u)\mathrm{d}u\\&=&
\int_{-a}^a\int_0^f\cosh^2\ln\left|\exp\left(\frac{2\pi}{f}w+d_0\right)\right|\mathrm{d}v\mathrm{d}u\\
&=&{\rm Area}(\mathcal{C}(2\pi/f;d_0)\cap\mathrm H^a_{-a}),
\end{eqnarray*}
which proves \eqref{area} with $\mathcal{W}=\mathcal{C}(2\pi/f;d_0)\cap\mathrm H^a_{-a}$. Note here that
\begin{eqnarray}
{\rm Area}(\mathcal{C}(2\pi/f;d_0)\cap\mathrm H^a_{-a})&=&\int_{-a}^af\cosh^2\left(\frac{2\pi}{f} u+d_0\right)\mathrm{d}u\nonumber\\
&\geq&\int_{-a}^af\cosh^2\left(\frac{2\pi}{f} u\right)\mathrm{d}u\label{d0}\\
&=&{\rm Area}(\mathcal{C}(2\pi/f;0)\cap\mathrm H^a_{-a}).\nonumber
\end{eqnarray}

Now we need to find the catenoidal waist in $\mathrm H^a_{-a}$ with smallest area. This was already proved by Bernstein-Breiner \cite{1}, but is also proved here for completeness. This is a straightforward computation.

We observe that, by \eqref{d0}, a least area catenoidal waist is of the form $\mathcal{C}(\lambda;0)\cap\mathrm H^a_{-a}$ for some $\lambda>0$, i.e., symmetric with respect to the plane $\{z=0\}$. So we have only to consider catenoidal waists of this form.
The catenoid $\mathcal{C}(\lambda;0)$ has $2\pi/\lambda$ as the vertical component of the flux along a horizontal circle. Hence
$$A(\lambda):={\rm Area}(\mathcal{C}(\lambda;0)\cap\mathrm H_{-a}^a)
=\frac{2\pi}\lambda\int_{-a}^a\cosh^2(\lambda u)\mathrm{d}u=\frac{2\pi a}\lambda+\frac{\pi}{\lambda^2}\sinh (2\lambda a),$$
and
$$A'(\lambda)=-\frac{2\pi a}{\lambda^2}-\frac{2\pi}{\lambda^3}\sinh(2\lambda a)+\frac{2a\pi}{\lambda^2}\cosh(2\lambda a).$$
Since $A'(\lambda)=0$ has a unique solution, ${\rm Area}(\mathcal{C}(\lambda;0)\cap\mathrm H_{-a}^a)$ must have a unique minimum. At that minimum we can easily show that $\lambda$ satisfies
$$\tanh (\lambda a)=\frac{1}{\lambda a}.$$
Then \eqref{beta} implies that $\beta=\lambda a$ and hence
\begin{equation}\label{ew}
{\rm Area}(\mathcal{C}(2\pi/f;0)\cap\mathrm H^a_{-a})\geq{\rm Area}(\mathcal{C}(\beta/a;0)\cap\mathrm H^a_{-a}).
\end{equation}

Since the central waist circle of $\mathcal{C}(\beta/a;0)$ has radius $a/\beta$, we have
$$\mathcal{C}(\beta/a;0)=\frac{a}{\beta}\,\mathcal{C}.$$
Therefore
$$\mathcal{C}(\beta/a;0)\cap\mathrm H_{-a}^a=\frac{a}{\beta}\,\mathcal{C}^\beta_{-\beta},$$
which, together with \eqref{d0} and \eqref{ew}, gives \eqref{area2}. Clearly $\frac{a}{\beta}\,\mathcal{C}^\beta_{-\beta}$ is maximally stable. If \begin{equation}\label{abeta}
{\rm Area}(\Sigma)={\rm Area}\left(\frac{a}{\beta}\,\mathcal{C}^\beta_{-\beta}\right),
\end{equation}
then all the inequalities in \eqref{hk}, \eqref{d0}, \eqref{ew} should become equalities. Thus
$$\Sigma=\frac{a}{\beta}\,\mathcal{C}^\beta_{-\beta}.$$
Similarly
$${\rm Area}(\Sigma)={\rm Area}(\mathcal{W})$$
if and only if $\Sigma=\mathcal{W}=\mathcal{C}(2\pi/f;d_0)\cap\mathrm H^a_{-a}$.
This completes the proof.
\end{proof}
\vspace{.3cm}

In conclusion, we would like to propose the following:\\

\noindent{\bf Problems.}

{\bf 1.} Does there exist a minimal surface with two boundary components in a slab which has a horizontal section with mixed orientations? (See Figure 1.)

{\bf 2.} Let $\Sigma\subset\mathbb R^3$ be an immersed minimal annulus in a slab ${\rm H}$ such that $\Sigma\cap P$ is a figure-eight curve for any horizontal plane $P\subset {\rm H}$. And let $\mathcal{W}$ be the least area maximally stable catenoidal waist in ${\rm H}$. Is it true that ${\rm Area}(\Sigma)\geq{\rm Area}(\mathcal{W})$? More generally, is it possible to remove in Theorem \ref{theorem} the hypotheses on the orientability and on the rotation numbers of the level sets of the surface?

{\bf 3.} Given a minimal hypersurface $\Sigma$ in a slab $\mathrm H$ of $\mathbb R^n$, $n\geq4$, show that its volume is bigger than that of an $(n-1)$-dimensional catenoid in $\mathrm H$ (see \cite{9} for the higher dimensional catenoid), or their volumes are equal if and only if $\Sigma$ is the maximally stable catenoidal waist symmetric with respect to the mid-hyperplane of $\mathrm H$.

{\bf 4.} Given a minimal surface $\Sigma$ in a slab $\mathbb H^2\times[-a,a]$ of the homogeneous manifold $\mathbb H^2\times\mathbb R$, where $\mathbb H^2$ denotes the hyperbolic plane, prove a theorem similar to Theorem \ref{theorem}.

\bibliographystyle{plain}
\bibliography{slab}

\end{document}